\documentclass[10pt]{amsart}
\usepackage[margin=1.2in,marginparsep=0.1in,marginparwidth=1in]{geometry}
\usepackage{amssymb,amsmath,amsthm,amstext,amscd,latexsym,graphics,graphicx,bbm,caption}
\usepackage[usenames,dvipsnames,svgnames,table]{xcolor}
\usepackage[plainpages=false,colorlinks=true]{hyperref}
\usepackage{tikz}
\usepackage{tikz-cd}
\usetikzlibrary{positioning}

\tikzset{>=stealth}
\hypersetup{citecolor=Sepia,linkcolor=blue, urlcolor=blue}

\usepackage{cite}   

\makeatletter
\def\@tocline#1#2#3#4#5#6#7{\relax
  \ifnum #1>\c@tocdepth 
  \else
    \par \addpenalty\@secpenalty\addvspace{#2}%
    \begingroup \hyphenpenalty\@M
    \@ifempty{#4}{%
      \@tempdima\csname r@tocindent\number#1\endcsname\relax
    }{%
      \@tempdima#4\relax
    }%
    \parindent\z@ \leftskip#3\relax \advance\leftskip\@tempdima\relax
    \rightskip\@pnumwidth plus4em \parfillskip-\@pnumwidth
    #5\leavevmode\hskip-\@tempdima
      \ifcase #1
       \or\or \hskip 2em \or \hskip 2em \else \hskip 3em \fi%
      #6\nobreak\relax
    \dotfill\hbox to\@pnumwidth{\@tocpagenum{#7}}\par
    \nobreak
    \endgroup
  \fi}
\makeatother

\newtheorem{intro-thm}{Theorem}[]
\theoremstyle{plain}
\newtheorem{thm}{Theorem}[section]
\newtheorem{theorem}[thm]{Theorem}

\newtheorem{lemma}[thm]{Lemma}
\newtheorem{corollary}[thm]{Corollary}
\newtheorem{proposition}[thm]{Proposition}

\theoremstyle{definition}
\newtheorem{remark}[thm]{Remark}

\newtheorem{example}[thm]{Example}

\newcommand{\Spec}{{\rm Spec \,}}

\renewcommand{\tilde}{\widetilde}

\newcommand{\sD}{{\mathcal D}}

\newcommand{\sO}{{\mathcal O}}

\newcommand{\sS}{{\mathcal S}}

\newcommand{\A}{{\mathbb A}}

\newcommand{\F}{{\mathbb F}}

\renewcommand{\P}{{\mathbb P}}

\newcommand{\Z}{{\mathbb Z}}

\newcommand{\ra}{\rightarrow}
\newcommand{\sms}{\sS m_S}
\newcommand{\sspectra}{Spt_{S^1}(\sS m_S)}
\newcommand{\xnisspectra}{Spt_{S^1}(X_{Nis})}
\let\lim=\relax
\DeclareMathOperator*{\lim}{lim}
\let\dim=\relax

\DeclareMathOperator{\dim}{dim}
\DeclareMathOperator{\codim}{codim}
\DeclareMathOperator*{\colim}{colim}

\input{xy}
\xyoption{all}

\begin{document}
\title{A Nisnevich Local BLOCH-OGUS THEOREM over a general base}
\author{Neeraj Deshmukh}
\author{Girish Kulkarni}
\author{Suraj Yadav}



\date{}

\begin{abstract}
 We prove the exactness of the Nisnevich Gersten complex over a Noetherian irreducible base of finite type under some conditions. We also obtain, as a consequence, a Nisnevich analogue of the Bloch-Ogus
theorem for \'{e}tale cohomology in this setting. 
\end{abstract}

\maketitle
\section{Introduction}

\noindent Given a smooth algebraic variety $X$ over a field, the classical Bloch-Ogus theorem says that the Gersten complex is exact for \'{e}tale cohomology with coefficients in the twisted sheaf $\mu_n^{\otimes i}$ of $n$-th roots of unity. Originally proved by Bloch and Ogus in \cite{Bloch-Ogus}, it was extended by Gabber \cite{Gab} to any torsion sheaf on $X_{\text{\'{e}t}}$ which comes from the base field. In fact, the methods in \cite{Gab} could be applied to any cohomology theory with supports that has the same properties as \'{e}tale cohomology. This was done in \cite{colliot} by Colliot-Th\'{e}l\`{e}ne, Hoobler and Kahn. Using the ideas of Gabber, they were able to show that for any $\A^1$-invariant cohomology theory $E$ for smooth varieties over a field, the associated Gersten complex is exact. The essence of their methods lies in a geometric presentation lemma due to Gabber \cite[Theorem 3.1.1]{colliot}.

In \cite{schmidt2019}, Strunk and Schmidt prove a Nisnevich local analog of the Bloch-Ogus theorem for discrete valuation rings with only infinite residue fields.
They adapt the results in \cite{colliot} to the mixed characteristic setting using a Nisnevich local version of the geometric presentation lemma for discrete valuation rings with only infinite residue fields (see \cite[Theorem 2.1]{schmidt2018}). In this approach to the Bloch-Ogus theorem over more general base schemes, the geometric presentation lemma plays a crucial role. The geometric presentation lemma has been extended to all Noetherian domains with only infinite residue fields in \cite{deshmukh2019}.  A generalisation with no restriction on residue fields has been proved in \cite{druzhinin2019}. While the version in \cite{druzhinin2019} has no restriction on the base scheme, the conclusion obtained is slightly less general in comparison (see Remark \ref{diff}). However, it turns out to be sufficient in the present context. 

In this note, we extend the theorems in \cite{schmidt2019} to Noetherian irreducible base of finite type using the presentation lemma as in \cite{druzhinin2019}. Our main result is the following Nisnevich local generalisation of the Bloch-Ogus theorem (see Section \ref{BOsec} for notation):

\begin{theorem}\label{cor: Bloch Ogus for etale cohomology of Nisnevich local schemes}
	Let $S$ be a $J$-$2$ Noetherian irreducible regular scheme of finite type. Fix a point $s \in S$.
	Let $X/S$ be smooth of finite type, $d = {\rm dim}(X)$ and $C^\bullet$ an l.c.c.~complex in $\sD_c^b(S_{\rm et},\Lambda)$.
	Let $x$ be a point of $X$ lying over $s$ and $Y = X_x^{h}$ the Nisnevich local scheme at $x$.
	Then there is an exact sequence
	\begin{multline}\label{BOres}
		0 \to {\rm H}^n(Y_{{\rm et}},C^\bullet\vert_Y) \xrightarrow{e} \bigoplus_{z \in Y^{(0)}} {\rm H}^n(k(z),z^\ast C^\bullet\vert_Y) \xrightarrow{d^{0}} \cdots\\
		\cdots \xrightarrow{d^{d-1}} \bigoplus_{z \in Y^{(d)}} {\rm H}^{n-d}(k(z),z^\ast C^\bullet\vert_Y(-d)) \to 0 .
	\end{multline}
\end{theorem}

In fact, we prove a more general result about the Gersten resolution of a cohomology theory (Theorem \ref{gersternfornisnevich}).  To that end, we follow the methods in \cite{schmidt2019}. The important distinction is that we replace the presentation lemma \cite[Theorem 2.1]{schmidt2018} with the more general result \cite[Remark 3]
{druzhinin2019}. We prove the following theorem (see Section \ref{sec:preliminaries} for notation):

\begin{theorem}\label{gersternfornisnevich}
	Let $S$ be a Noetherian irreducible scheme of finite type of dimension $p$ and let $E \in Spt_{S^1}(Sm_S)$ be a $\A^1$-Nisnevich local fibrant spectrum and $X \in Sm_S$ of dimension $d$. Then the complex
	\begin{multline}\label{Gerstenres}
		0 \rightarrow (E_{X}^{n})^{\sim} \xrightarrow{\text{e}} \bigoplus_{{z \in X^{(0)}}}  \mathfrak{j}_* \mathfrak{j}^* E^{n}_{Z/X}  \xrightarrow{d^{0}} \bigoplus_{{z \in X^{(1)}}} \mathfrak{j}_* \mathfrak{j}^* E^{n+1}_{Z/X} \xrightarrow{d^{1}} \cdots \\ \cdots \xrightarrow{d^{d-2}} \bigoplus_{{z \in X^{(d-1)}}}    \mathfrak{j}_* \mathfrak{j}^* E^{n+d-1}_{Z/X} \xrightarrow{d^{d-1}}  \bigoplus_{{z \in X^{(d)}}}   \mathfrak{j}_* \mathfrak{j}^* E^{n+d}_{Z/X} \rightarrow 0 
	\end{multline}
	is exact with possible exceptions at $(E_{X}^{n})^{\sim} $ and $\bigoplus_{{z \in X^{(i)}}}  \mathfrak{j}_* \mathfrak{j}^* E^{n+i}_{Z/X}$ for $1\leq i \leq p$.	Furthermore, the above complex is exact everywhere if for each $x \in X$ which lies over $s\in S$ and for any irreducible closed subset $Z\subset X$ of codimension $k$ satisfying either
	\begin{enumerate}
		\item $X_{x} \subseteq Z \subset X$ or
		\item $Z$ is an irreducible component of $X_x$
	\end{enumerate}
	there exists $Z' \supset Z$ of codimension $k-1$ such that following (forget support) map is trivial 
	\[E_{Z/X}(X^{h}_{x}) \rightarrow E_{Z'/X}(X^{h}_{x}).\]
	Here $x$ is the closed point of the henselisation of $S$ at $s$. In fact, this gives us a resolution of $(E_{X}^{n})^{\sim} $ by flabby Nisnevich sheaves, which implies the following isomorphism
	\[H^{k}(Y, (E_{X}^{n})^{\sim} )   \cong H^{k}(\mathcal{G}^{\bullet}(E,n)(Y)) \]
	for $Y \in X_{Nis}$, which vanishes for $k > d$.
\end{theorem}

Specialising Theorem \ref{gersternfornisnevich} to \'{e}tale cohomology gives us the Bloch-Ogus Theorem for Nisnevich local schemes (Theorem \ref{cor: Bloch Ogus for etale cohomology of Nisnevich local schemes}). Note that both Theorem \ref{cor: Bloch Ogus for etale cohomology of Nisnevich local schemes} and Theorem \ref{gersternfornisnevich} generalise to higher dimensions the dimension one case of $S$ a Dedekind scheme proved in \cite[Corollary 6.10 and Theorem 5.12]{schmidt2019}, respectively.

While Theorem \ref{gersternfornisnevich} is a fairly straightforward generalisation of \cite[Theorem 5.12]{schmidt2019}, concluding Theorem \ref{cor: Bloch Ogus for etale cohomology of Nisnevich local schemes} from this is a bit subtle compared to the dimension one case in \cite{schmidt2019}. Indeed, it requires the full strength of Gabber's absolute purity theorem \cite{fujiwara} in contrast with \cite{schmidt2019}, where only absolute purity for closed subschemes in the special fiber suffices. This is where the regular and $J$-2 hypotheses on $S$ come in. These are two technical conditions needed to employ Gabber's absolute purity theorem. These assumptions on $S$ are not unreasonable as they are implicit in the dimension one case in \cite{schmidt2019} (a Dedekind local ring is regular and $J$-$2$). Note, however, that Theorem \ref{gersternfornisnevich} holds for any Noetherian irreducible scheme $S$ of finite dimension.

A scheme $S$ is said to be $J$-2 if for any finite type scheme $X$ over $S$ the regular locus of $X$ is open. All fields, $\Z$, Noetherian complete local rings, or schemes of finite type over these rings are $J$-2. All (quasi-)excellent schemes are $J$-2.

The regularity of $S$ ensures that $X$ is regular (in the absolute sense) while $J$-2 ensures that for any irreducible closed subscheme $Z$ of $X$ the regular locus is open in $Z$. As the cohomology groups in the Gersten resolution (\ref{BOres}) are defined as colimits over open neighborhoods of the generic point of $Z$, to prove Theorem \ref{cor: Bloch Ogus for etale cohomology of Nisnevich local schemes} it suffices to have absolute purity for the regular locus of $Z$. The argument is developed in Section \ref{BOsec}.\\

\noindent \textbf{Outline.} We begin with some preliminaries in Section \ref{sec:preliminaries} about model structures on spectra over smooth schemes, and set up the notation and terminology. In Section \ref{main}, we prove Theorem \ref{gersternfornisnevich} as well as discuss some examples where the theorem fails to hold. Finally, we prove Theorem \ref{cor: Bloch Ogus for etale cohomology of Nisnevich local schemes} in Section \ref{BOsec}.\\

\noindent \textbf{Acknowledgments.} The first-named author was supported by the INSPIRE fellowship of the Department of Science and Technology, Govt.\ of India during this work. The second-named author was supported by DFG SPP 1786 grant for his stay at Bergische Universit\"at Wuppertal.  The last-named author was supported by the NBHM fellowship of the Department of Atomic Energy, Govt.\ of India during this work.
We thank Amit Hogadi for his comments. We also thank Matthias Wendt for his many helpful comments and the suggestion that Ayoub's counterexample to stable $\A^1$-connectivity might also provide a counterexample to Theorem \ref{gersternfornisnevich}. We thank anonymous referee for comments and suggestions.

\section{Preliminaries and Notation}\label{sec:preliminaries}
\noindent We will briefly review the set-up required to prove Theorem \ref{gersternfornisnevich}. There is no claim at originality of content or presentation and most of the material can be found in \cite{schmidt2019}. We reproduce it here to introduce the notation and for the sake of clarity of exposition.

\noindent Let $S$ be a Noetherian base scheme of finite type. Denote by $Sm_S$ the category of smooth schemes of finite type over $S$ and by $\sspectra$ the category of presheaves of spectra on $Sm_S$. For $X\in Sm_S$, $\xnisspectra$ is the category of presheaves of spectra on the small Nisnevich site of $X$. We will work with the object-wise model structures on these categories.

\noindent A morphism $f:X\rightarrow Y$ in $Sm_S$ induces a morphism of the corresponding sites, $f:Sm_Y\rightarrow Sm_X$ by pullback. This gives rise to a Quillen adjunction,
\[f^*:Spt_{S^1}(Sm_Y)\leftrightarrows Spt_{S^1}(Sm_X):f_*\]
and on the small Nisnevich sites,
\[\xnisspectra\mathrel{\mathop{\rightleftarrows}^{f_{\#}}_{f^*}} Spt_{S^1}(Y_{Nis})\mathrel{\mathop{\rightleftarrows}^{f^*}_{f_*}}\xnisspectra\]
where for the first one we have to assume that $f$ is an object of $Y_{Nis}$ while the second one always exists.
Given an $E\in Spt_{S^1}(Sm_S)$, we denote by $E_X$ its restriction to $\xnisspectra$ and $E^n(X):=\pi_{-n}(E(X))$  (see \cite[Definition 2.3]{schmidt2019}). Let $Z\subseteq X$ be a closed subset and consider the open immersion $j: X\setminus Z\rightarrow X$. Then the unit of adjunction associated to the map $j$ induces a canonical map 
\[\eta_j: E_X\rightarrow j_* j^*E_X\]
in $\xnisspectra$. We denote by $E_{Z/X}$, the homotopy fiber of $\eta_j$ in $\xnisspectra$. Moreover,if $Z\subseteq Z'\subseteq X$ are closed subsets of $X$, we have a canonical map $E_{Z/X}\rightarrow E_{Z'/X}$ called \textit{forget support} map (see \cite[Lemma 3.7]{schmidt2019}).




\noindent Recall from \cite{morel1999}, that a Nisnevich distinguished square is a pullback square

		\begin{center}
		\begin{tikzcd}
		& V \arrow[d] \arrow[r]&Y\arrow[d,"f"]\\
		& U\arrow[r,"i"]&{X}\\
		\end{tikzcd}
	\end{center}
such that $i: U\rightarrow X$ is an open immersion, $f$ is an \'{e}tale morphism of finite type and $(X\setminus i(U))_{red}\times_X Y\rightarrow X\setminus i(U))_{red}$ is an isomorphism. A spectrum $E \in \sspectra $ is called Nisnevich local fibrant if and only if $E(\emptyset)=*$ and for each Nisnevich distinguished square $P$, $E(P)$ is a homotopy pullback square. Furthermore by \cite[Lemma 3.11]{schmidt2019} an objectwise fibrant spectrum $E\in \sspectra$ is Nisnevich local fibrant if and only if for all Nisnevich distinguished squares as above, the induced morphism
\[E_{Z/X}\rightarrow f_*f^*E_{Z/X}\simeq f_*E_{f^{-1}(Z)/Y}\]
is an equivalence, where, $Z=X\setminus U$.

For a sprectrum $E_X\in Spt_{S^1}(X_{Nis})$ recall from \cite[Definition 3.1]{schmidt2019}, \[E_{X^{(p)}}:=\underset{\substack{Z\subset X \text{closed}\\codim(Z,X)\geq p}}{hocolim}E_{Z/X}\]
in $Spt_{S^1}(X_{Nis})$, for $p \geq 0$. The structure maps are given by forget support maps described in preceding paragraph.

\noindent For a smooth scheme $X$ of relative dimension $d$, by the universal property of colimits and the definition of codimension, we automatically have a filtration,
\[E_{X^{(d)}}\rightarrow E_{X^{(d-1)}}\rightarrow\ldots\rightarrow E_{X^{(0)}}=E_X\]
of presheaves of spectra on $X_{Nis}$. We denote by $E_{X^{(p/p+1)}}$, the homotopy cofiber of the map $E_{X^{(p+1)}}\rightarrow E_{X^{(p)}}$. This cofiber sequence gives rise to a long exact sequence of homotopy groups for each $p$. Using these long exact sequences for each $p$, we can construct a chain complex of presheaves of abelian groups on $X_{Nis}$,
\begin{equation}\label{leq}
0\rightarrow E^n_X\overset{e}{\rightarrow} E^n_{X^{(0/1)}}\overset{d^0}{\rightarrow} E^{n+1}_{X^{(1/2)}}\overset{d^1}{\rightarrow} \ldots\overset{d^{d-2}}{\rightarrow} E^{n+d-1}_{X^{(d-1/d)}}\overset{d^{d-1}}{\rightarrow} E^{n+d}_{X^{(d)}}\rightarrow 0
\end{equation}


\noindent 
 The complex is exact if following morphisms are all zero
\begin{center}
 $E^n_{X^{(1)}} \rightarrow E^n_{X^{(0)}} $ 
 
  $ E^{n+k}_{X^{(k)}} \rightarrow E^{n+k}_{X^{(k-1)}} $
  
    $ E^{n+k+1}_{X^{(k+2)}} \rightarrow E^{n+k}_{X^{(k+1)}} $
\end{center}
we refer the reader to \cite[\S 4]{schmidt2019}.
\noindent One can obtain similar conditions after sheafifying the above complex. This observation leads to \cite [Proposition 4.6]{schmidt2019}, which will be used in the proof of Theorem \ref{gersternfornisnevich}. In the interest of brevity we don't state the proposition here.\\
Using \cite[Proposition 3.19, Corollary 3.20]{schmidt2019}(See also \cite[Lemma 1.2.2]{colliot}), for a Nisnevich local fibrant spectrum $E\in \sspectra$, we may rewrite (\ref{leq}) as
\begin{multline*}
0\rightarrow E^n_X\rightarrow \underset{z\in X^{(0)}}\bigoplus \mathfrak{j}_* \mathfrak{j}^* E^n_{Z/X}\rightarrow \underset{z\in X^{(1)}}\bigoplus \mathfrak{j}_* \mathfrak{j}^* E^{n+1}_{Z/X} \rightarrow \ldots\\ \ldots \rightarrow \underset{z\in X^{(d-1)}}\bigoplus \mathfrak{j}_* \mathfrak{j}^* E^{n+d-1}_{Z/X} \rightarrow  \underset{z\in X^{(d)}}\bigoplus \mathfrak{j}_* \mathfrak{j}^* E^{n+d}_{Z/X}\rightarrow 0
\end{multline*} where $z \in X^{(p)}$ is a point, $Z:=\overline{\lbrace z\rbrace}$ and $\mathfrak{j}$ is the canonical morphism  $\Spec(\sO_{X,z})\rightarrow X$. In fact more is true. One can also show that these sheaves of abelian groups are \textit{flabby} \cite[Corollary 3.23]{schmidt2019}. This leads us to the following definition of \textit{Nisnevich Gersten complex} of $E$, denoted as $\mathcal{G}^{\bullet}(E,n)$ \cite[Definition 4.3]{schmidt2019} where,
\[\mathcal{G}^{p}(E,n):=\underset{z\in X^{(p)}}\bigoplus \mathfrak{j}_* \mathfrak{j}^* E^n_{Z/X}.\] We call $\mathcal{G}^{\bullet}(E,n)$ as the \textit{Nisnevich Gersten complex} of $E$ and homotopical degree $n$.



\section{Proofs of Theorems}\label{main}
\noindent In this section, we prove our theorems. We first prove Theorem \ref{gersternfornisnevich}, which is a general result about $\A^1$-Nisnevich local fibrant spectrum $E$ providing conditions for exactness of the Nisnevich Gersten complex defined in previous section. As an application of this result, when $E$ is taken to be  Nisnevich sheafification of \'{e}tale cohomology, we prove Theorem \ref{blochogusetale}. The Bloch-Ogus theorem for \'{e}tale cohomology of Nisnevich local schemes immediately follows (Theorem \ref{cor: Bloch Ogus for etale cohomology of Nisnevich local schemes}). 

\noindent Throughout this section let $S$ be a Noetherian irreducible scheme of finite type. 

\subsection{Gersten complex for  $\A^1$-Nisnevich local fibrant spectrum}\label{effa}
In this section, we establish Theorem \ref{gersternfornisnevich}. This theorem gives the exactness condition for the Nisnevich Gersten complex associated to an $\A^1$-invariant cohomology theory with Nisnevich descent for smooth schemes over $S$.\\

\noindent For $E\in \sspectra$ and a morphism $f: Y\to X$ in $\sms$, induces the map
\[\eta_f :E_X \to f_*E_{Y}\] in $\xnisspectra$. Furthermore, for a closed subset $Z\subset X$, pullback $\tilde{Z}$ of $Z$ along $f$ and a pullback diagram
\begin{center}
	\begin{tikzcd}
		Y\setminus \tilde Z \arrow[r, hook, "\tilde j"]\arrow[d, "\tilde f"]&Y \arrow[d, "f"]&\\
		X\setminus Z \arrow[r, hook, "j"] &X        
	\end{tikzcd}
\end{center}
we can define the morphism 
\[\eta_f :E_{Z/X} \to f_*E_{\tilde{Z}/Y}\] for details see \cite[Constuction 5.3]{schmidt2019}.\\

\noindent The following lemma is essentially \cite[Corollary 3 and Remark 3]{druzhinin2019}. This lemma provides the required Nisnevich distinguished square, which usually is a consequence of Gabber's presentation lemma. 
\begin{lemma}{\cite[Remark 3]{druzhinin2019}}\label{gabberforus}
	Let $X$ be an essentially smooth henselian local scheme over a scheme $S$ and let $Z \subset X$ be a closed subscheme of positive relative codimension. Then there is a map $p:X \to \A^{1}_{V}$, where $V = {(\A^{\dim X-1})}^{h}_{0}$ is the henselisation at the point $0$, such that $p$ is \'{e}tale, $p$ induces an isomorphism $Z \simeq p(Z)$, and $p(Z)$ is finite over $V$. Consequently, giving the following Nisnevich distinguished square:
	
	\begin{center}
		\begin{tikzcd}
			& X\setminus Z \arrow[d] \arrow[r]&X\arrow[d,"f"]\\
			& \A^{1}_{V} \setminus p(Z) \arrow[r]&{\A^{1}_{V}}
		\end{tikzcd}
	\end{center}
\end{lemma}

\begin{remark} \label{diff}
	Note that in the Nisnevich distinguished square in the above lemma $V$ is a limit of Nisnevich neighborhood of $ \A^{\dim X-1}$, whereas in \cite{schmidt2018} and \cite{deshmukh2019} it is a Zariski neighborhood in  $ \A^{\dim X-1}$.
	
\end{remark}
\noindent The following proposition generalises \cite[Proposition 5.9]{schmidt2019} to a more general base. The proof is exactly the same, except for the input from the presentation lemma.

\begin{proposition}\label{forgetsupportistrivial}
	Let $E \in Spt_S^1(Sm_s)$ be a $\A^1$-Nisnevich local fibrant spectrum. Let $X\in \sms$ be irreducible scheme, $Z\hookrightarrow X$ be a closed subscheme and $x$ be a point in $Z$ lying above $s\in S$, such that $\dim(Z_s)< \dim(X_s)$. Then Nisnevich-locally around $x$ there exist 
	
	\begin{enumerate}
		\item $V \in \sms$  a smooth relative curve $p:X\ra V$ with $Z$ finite over $V$  
		\item a closed subscheme $Z'\hookrightarrow X$ containing $Z$ such that $\codim(Z',X)=\codim(Z,X)-1$.
	\end{enumerate}
	and the forget support map induces the trivial morphism
	$$p_*E_{Z/X}\ra p_*E_{Z'/X}$$
	in the homotopy category. In particular $ E_{Z/X}(X)\ra E_{Z'/X}(X)$ is trivial.
\end{proposition}
\begin{proof}
	From Lemma \ref{gabberforus} (and using a standard limiting argument, see \cite{DG} IV \S{8}) we can find a Nisnevich distinguished square
	\begin{center}
		\begin{tikzcd}
			& X\setminus Z \arrow[d] \arrow[r]&X\arrow[d,"f"]\\
			& \A^{1}_V \setminus f(Z) \arrow[r]&{\A^{1}_V}
		\end{tikzcd}
	\end{center}
	such that $Z\hookrightarrow X\xrightarrow{f}\A^1_V\xrightarrow{\pi}V$ is finite, after possibly shrinking $X$ Nisnevich locally around $x$. Let $p= \pi\circ f$, $\overline{Z}= p(Z)_{red}$ and $Z'=p^{-1}(\overline{Z})$. Since $\pi$ and $f$ are flat, so is $p$ hence it follows that $\codim(Z',X)=\codim(Z,X)-1$. 
	By the excision \cite[Lemma 3.11]{schmidt2019} it follows that the upper horizontal morphism in the following diagram 
	\begin{center}
		\begin{tikzcd}
			& E_{f(Z)/ \A^1_V} \arrow[d] \arrow[r,"\simeq"]&f_*E_{Z/ X}\arrow[d,"f"]\\
			& E_{\A^{1}_Z / \A^1_V}\arrow[r]&f_*E_{Z'/ X}
		\end{tikzcd}
	\end{center}
	is an equivalence.  In the above diagram the vertical maps are respective forget support maps and $f^{-1}f(Z)=Z$. Applying $\pi_*$ to the above diagram we get the following diagram:
	\begin{center}
		\begin{tikzcd}
			& \pi_*E_{f(Z)/ \A^1_V} \arrow[d] \arrow[r,,"\simeq"]&p_*E_{Z/ X}\arrow[d,"f"]\\
			&\pi_* E_{\A^{1}_Z / \A^1_V}\arrow[r]&p_*E_{Z'/ X}
		\end{tikzcd}
	\end{center}

	From \cite[Lemma 5.8]{schmidt2019}, the left vertical map is trivial. Hence the right vertical map is also trivial thereby proving the proposition.
\end{proof}
\begin{corollary}\label{forgetsupportgeneric}
	Under the assumptions of the previous proposition, the forget support map
	\[E_{Z/X}(X_{x,\eta}^h)\rightarrow E_X(X_{x,\eta}^h)\]
	is trivial, $X_{x,\eta}^h$ is the generic fiber of the Henselian local scheme at $x$  and $ E_X(X_{x,\eta}^h)$ denotes stalk of $E$ at $x$ of $X$.
\end{corollary}

\begin{proof}
	By Lemma \ref{gabberforus}, we can find a cofinal family of Nisnevich neighbourhoods $(W,w)$ of $x$ each admitting a Nisnevich distinguished square as in Proposition \ref{forgetsupportistrivial}.	Since, $E(X_{x,\eta}^h)$ is   $\colim_{(W,w)}E(W_{\eta})$, where $W_{\eta}$ is the generic fiber it is sufficient show that for such neighbourhoods the forget support map is trivial. So we assume $W=X$. As $X_\eta =  \underset{X_\eta  \subset T \subset X}{colim} T$ and $Z_\eta =  \underset{X_\eta  \subset T  \subset X}{colim} T \cap Z$, where $T$ is open subscheme of $X$, we have the following distinguished square
	
	\begin{center}
		\begin{tikzcd}
			T\setminus Z \cap T\arrow[r]\arrow[d]&	T\arrow[d]\\
			\A^1_{V}\setminus f(Z\cap T)\arrow[r]&	\A^1_{V}\\
		\end{tikzcd}
	\end{center}
	Now by previous proposition $E_{Z/X}(T)\ra E_{X}(T)$ is trivial. Hence $E_{Z/X}(X_\eta)\ra E_{X}(X_\eta)$ is trivial , as $ E_{X}(X_\eta) = \colim_T E_{X}(T)$. In a similar fashion $E_{Z/X}(X_{x,\eta}^h)\rightarrow E_X(X_{x,\eta}^h)$ is trivial.

\end{proof}

\noindent We now prove Theorem \ref{gersternfornisnevich}.

\begin{proof}[Proof of Theorem \ref{gersternfornisnevich}]
	As we can check exactness  stalkswise, we assume $S$ to be  spectrum of a Henselian local ring. Let $\sigma$ be the closed point.  By  \cite[Proposition. 4.6(2)(ii,iii)]{schmidt2019} the theorem  follows by showing for a given closed subscheme $Z \subset X$ of codimension $\geq p+1$, there exists $Z \subseteq Z' \subseteq X$ with $\codim( Z', X ) < \codim( Z, X )$, such that forget support map $E_{Z/X}^{n+s}(X^{h}_{x}) \rightarrow E_{Z'/X} ^{n+s} (X^{h}_{x})$ is trivial.  We can assume $X$ to be a Henselian local scheme.

	If $Z$ does not contain the special fibre $X_\sigma$, then by Proposition  \ref{forgetsupportistrivial} we are done. So now suppose $Z$ contains the special fibre. If $Z$ is irreducible, then by hypothesis there is a $Z'$ such that $\codim(Z',X)< \codim(Z,X)$ and the forget support map $E_{Z/X}(X)\rightarrow E_{Z'/X}(X)$  is trivial. If $Z$ is not irreducible, then we can write $Z = \cup_{i} Z_i  $ where $Z_{i}$'s are the irreducible components of $Z$. Without loss of generality assume $i=2$.  Hence, by hypothesis (and in case one of the irreducible component doesn't entirely lie over the closed point of $S$, by Proposition \ref{forgetsupportistrivial}) there exist $ Z_1 \subset T_1$ and $ Z_2 \subset T_2$ such that  forget support maps $E_{Z_{1}/X} (X)\rightarrow E_{T_{1}/X}(X)  $ and $E_{Z_{2}/X}(X) \rightarrow E_{T_{2}/X} (X)$ are trivial. 
	
	Writing $T= T_1 \cup T_2$ we prove the forget support map $E_{Z/X}(X) \rightarrow E_{T/X}(X)$ is trivial. Note that as $E_{Z_{i}/X} (X)\rightarrow E_{T_{i}/X}(X)  $ is trivial so is the composition $E_{Z_{i}/X} (X)\rightarrow E_{T_{i}/X}(X) \rightarrow E_{T/X}(X)$, for $i=1,2$. Since we have the triangle $E_{Z/X}(X) \xrightarrow{\text{f}} E_{T/X}(X) \xrightarrow{\text{g}}  E_{(T\setminus Z)/(X\setminus Z)}(X \setminus Z)$, by a general fact about triangulated categories, proving $f$ is trivial  is equivalent to proving $g$ is a monomorphism.  Now using the isomorphism $E_{Z/X} (U)  \cong E_{(U\cap Z)/U}(U)$ for any open subscheme $U$ in $X$, we have $ E_{T/X}(X \setminus Z) \cong E_{(T\setminus Z)/(X\setminus Z)}(X \setminus Z) $. This implies that $g$ factors as $ E_{T/X}(X) \rightarrow E_{T/X}(X \setminus Z_1)   \rightarrow E_{(T\setminus Z)/(X\setminus Z)}(X \setminus Z)$. We will prove that both these morphisms are monomorphisms. 
	
	We have the following exact triangle for $Z_1$
	
	\begin{center}
		
		$E_{Z_{1}/X} (X)\rightarrow    E_{T/X}(X) \rightarrow     E_{(T\setminus Z_1)/(X\setminus Z_1)}(X \setminus Z_1) $
		
	\end{center}		
	Therefore, $  E_{T/X}(X) \rightarrow     E_{(T\setminus Z_1)/(X\setminus Z_1)}(X \setminus Z_1) \cong E_{T/X}(X \setminus Z_1)$ is a monomorphism. 
	
	Observing the triangle corresponding to $Z_2$
	\begin{center}
		$E_{Z_{2}/X} (X\setminus Z_{1})\rightarrow E_{T/X}(X\setminus Z_{1}) \rightarrow E_{(T\setminus Z_2)/(X\setminus Z_2)}(X \setminus Z) $
	\end{center}
	we conclude that $E_{T/X}(X\setminus Z_{1}) \rightarrow E_{(T\setminus Z_2)/(X\setminus Z_2)}(X \setminus Z) \cong E_{(T\setminus Z)/(X\setminus Z)}(X \setminus Z)$ is a monomorphism. This proves that composition $g  :  E_{T/X}(X) \rightarrow E_{T/X}(X \setminus Z_1)   \rightarrow E_{(T\setminus Z)/(X\setminus Z)}(X \setminus Z)$ is a monomorphism.
	
\end{proof}

We can greatly simplify the condition for exactness of the Nisnevich Gersten complex in Theorem \ref{gersternfornisnevich} when $S$ is $J$-2. In this case, it suffices to check the triviality of the forget support maps for \textit{regular} irreducible closed subschemes. The following is the precise statement:

\begin{proposition}\label{perf}
	In the setting of Theorem \ref{gersternfornisnevich}  assume $S$ to be a $J$-$2$ ring. Then if for every regular irreducible closed subscheme $Z\subset X$ of codimension $k$ satisfying either
	\begin{enumerate}
		\item $X_{\sigma} \subseteq Z \subset X$ or
		\item $Z$ is an irreducible component of $X_\sigma$
	\end{enumerate}
	there exists $Z' \supset Z$ of codimension $k-1$ such that the forget support  map  $E_{Z/X}(X^{h}_{x}) \rightarrow E_{Z'/X}(X^{h}_{x})$ is trivial, the complex (\ref{Gerstenres}) of Theorem \ref{gersternfornisnevich} is exact at all places.
\end{proposition}

\begin{proof}
	As $S$ is $J$-2,  every closed subschme $\overline{z} = Z$ has an open neighbourhood $U$ contanining $z$ such that $U \cap Z = Z^{reg}$ is regular. Further $E_{Z/X} (X^h_{x})\cong E_{Z^{reg}/X}(X^h_{x})$ and we proceed in the same manner as in the  proof of previous theorem.
\end{proof}

\begin{remark}
	Note that as \cite{schmidt2019} deals with the case where the base is regular of dimension one, the condition for exactness becomes the triviality of forget support maps $E_{Z/X} \rightarrow E_{X}$, with $X_{\sigma} \subseteq Z \subset X$, for $Z$ of codimension one. Moreover, such a $Z$ can not be irreducible as it contains $X_\sigma$ and so can be written as a union of its irreducible components which are either contained in $X_{\sigma}$ or not (in that case Corollary  \ref{forgetsupportgeneric} applies). Hence, for a regular dimension one base (say, a DVR), we can further simplify the condition for exactness of the Nisnevich Gersten complex to the condition that the forget support map $E_{X_{\sigma}/X} \rightarrow E_{X}$ is trivial.
\end{remark}

\subsection{Some Examples} We now discuss some examples where forget support condition stated in Theorem \ref{gersternfornisnevich} fails.

\begin{example}
	It's easy to come up with $E \in Spt_{S^1}(Sm_X)$ which do not satisfy  the triviality of forget support  maps.  Let $ j: Z \hookrightarrow X$ be an irreducible regular closed subscheme of codimension 1 in $X$. Let $E'$ be a $\A^1$-Nisnevich local fibrant spectrum in $ Spt_{S^1}(Sm_Z)$. Then $E := j_{\ast} E'$ is $\A^1$-Nisnevich local fibrant spectrum in $ Spt_{S^1}(Sm_X)$, supported on $Z$ and it follows from definitions that forget support map $E_{Z/X} \rightarrow E$  is not trivial. \\
	However if the Gersten resolution of $(E^{'n})^{\sim}$ is exact, then by exactness of $j_{\ast}$, pushforward of such a Gersten resolution is exact.  Moreover by Leray spectral sequence, such a pushforward is an acyclic resolution. Hence Gersten resolution of $(E^{n})^{\sim}$ is also exact. This shows that the hypothesis of forget support map being trivial is sufficient but not necessary.
\end{example}

\begin{example}\label{ayo}
	
	We now give an example of a spectrum $E' \in Spt_{S^1}(Sm_S)$ for which the Gersten resolution (\ref{Gerstenres}) is not exact. In fact, Ayoub's counterexample to Morel's conjecture on $\A^1$-connectivity \cite{Ayoub} works for us. We give a brief description here.\\
	Fix a perfect field $k$. Let $\mathcal{ K}_1^M$ denote the Nisnevich sheaf (on smooth schemes over $k$) respresenting Milnor $K$-theory. This sheaf, in fact, has transfers and hence belongs to $\mathrm{DM_{eff}}(k)$. Let $S$ be a normal surface in $\P^3$ given by equation $w(x^3 - y^2z) + f(x,y,z) = 0$ with $f$ a general homogeneous degree 4 polynomial. Then $S$ is non singular outside the point $[0:0:0:1]$. Denote by $i: S \hookrightarrow \P^3_k$ the inclusion map and by $\pi: \P^3_k \rightarrow \Spec k$ the structure map of $\P^3_k$.\\
	We will consider $\mathcal{ K}_{1,S}^M := i^{!}\pi^{\ast}(\mathcal{ K}_1^M) \in \mathrm{DM_{eff}}(k)$. It follows from Section 3 of \emph{op.\ cit} that the Nisnevich sheafifcation(denoted $cl_S$) of the presheaf $ U \mapsto H^1_{Nis}(U, \mathcal{ K}_{1,S}^M)$ on $Sm_S$ is not strictly $\A^1$-invariant. In particular, it cannot be zero. Therefore, the Gersten resolution of $\mathcal{ K}_{1,S}^M$ is not exact.
	
 Next we construct an $\A^1$-local fibrant spectrum with $(E^{'0})^{\sim} \cong i^{!}\pi^{\ast}\mathcal{ K}_1^M $. As $\mathcal{ K}_1^M$ is an $\A^1$-invariant sheaf with transfers, it is also strictly $\A^1$-invariant. This implies that the associated Eilenberg-Maclane spaces $K(\mathcal{ K}_1^M, n)$ are $\A^1$-local for all $n \geq 0$. Therefore, the spectrum $E$ with $E_n := K(\mathcal{ K}_1^M, n)$ is an $\A^1$- Nisnevich local fibrant spectrum in $Spt_{S^1}(Sm_k)$ with $(E^{0})^{\sim} \cong \mathcal{ K}_1^M $. Moreover, $E' := i^{!}\pi^{\ast}(E)$ is also an $\A^1$-Nisnevich local fibrant spectrum in $Spt_{S^1}(Sm_S)$ because  $i^{!}$ and $\pi^{\ast}$ both preserve fibrant objects in our situation. Finally $(E^{'0})^{\sim} \cong i^{!}\pi^{\ast}\mathcal{ K}_1^M $.
	
\end{example}

\begin{remark}
While $S$ defined in the previous example is not regular, the same example shows exactness of Gersten resolution fails for $i_{\ast} (E') $ in $Spt_{S^1}(Sm_{\P^3_k})$. This provides us with a counterexample over a regular base.
\end{remark}

\section{Bloch-Ogus Theorem} \label{BOsec}
In this section, we specialise to the \'{e}tale cohomology and prove Theorem \ref{cor: Bloch Ogus for etale cohomology of Nisnevich local schemes}. The idea is to verify the conditions stated in Proposition \ref{perf} about the vanishing of forget support maps. To verify these conditions we use Gabber purity for \'{e}tale cohomology. As Gabber purity requires the schemes to be regular, we have to put some extra hypothesis on our base scheme such as regularity and $J$-$2$. Note that \cite{schmidt2019} assume their base to a DVR, hence the condition of regularity and $J$-$2$ is implicit in their hypothesis.

All cohomology groups in this section, unless specified otherwise, are \'{e}tale cohomology groups. We fix the following notation

\begin{enumerate}
	\item Let $X$ be an irreducible,  smooth scheme of finite dimension over $S$.
	\item Let  $\Lambda$ the group $\Z/n$ for $n$ an integer co-prime to $p={\rm char\ } {\F}$ and $\mu_{n}$ be the sheaf of $n$-th roots of unity. Then given any constructible sheaf $\mathcal{F}$ of $\Lambda$ module, $\mathcal{F}(c)$ denotes $ \mathcal{F} \otimes \mu_{n}^{\otimes c}$, for any $c \in \mathbb{Z}$.
	\item Let $\sD_c^b(X_{et},\Lambda)$ be the derived category of bounded(above and below) complexes for which all the cohomology sheaves are constructible sheaves of $\Lambda$-module.
	\item For convenience we will call a complex $C^{\bullet}\in \sD_c^b(X_{et},\Lambda)$ with locally constant cohomology sheaves $H^q(C^{\bullet})$ for all $q$ an l.c.c. complex. 
\end{enumerate}

Given a closed immersion $i :Z \hookrightarrow X$ of regular Noetherian schemes, of pure codimension $c$. Gabber purity tells us when the following morphism of \'{e}tale cohomology groups , for any sheaf $\mathcal{F}$  of locally constant $\Lambda$ modules 
\begin{center}
	$H^{r-2c}(Z, \mathcal{F}(-c)) \rightarrow H_{Z}^{r}(X, \mathcal{F})$
	
\end{center}

is an isomorphism. See \cite{fujiwara} for details. In particular if $c=1$ (i.e $Z$ is of codimesnion 1 in $X$), we have  $\Lambda \cong H^{0}(Z, \Lambda) \xrightarrow{\cong} H_{Z}^{2}(X, \Lambda(1)) \rightarrow H^{2}(X, \Lambda(1))  $. We thus obtain a morphism $cl_{Z/X}: H^{0}(Z, \Lambda) \rightarrow H^{2}(X, \Lambda(1)) $. Moreover $cl_{Z/X} (1) = c^1(Z)$ where $c^1:Pic(X) \rightarrow H^{2}(X, \Lambda(1)) $ comes from Kummer exact sequence. See chapter 23 of \cite{milneLEC} for details.

\begin{remark}\label{zero}
	Observe the isomorphism $Hom_{\mathcal{ D}^{b}(X_{et},\Lambda) }(\Lambda[-2], \Lambda(1)) \cong H^2(X, \Lambda)$. Therefore if a cohomology class $[c] \in H^2(X, \Lambda)$ is trivial then the corresponding morphism $ \Lambda[-2] \rightarrow \Lambda(1)$ is trivial in $\mathcal{ D}^{b}(X_{et},\Lambda) $. For instance if $X$ is a local scheme then $Pic(X)$ is trivial and so $cl_{Z/X}(1)$ will be zero in $H^2(X, \Lambda)$. We will use this in proof of Lemma \ref{crux}.
\end{remark}

Since $i^!$ is a left exact functor, it induces a right derived functor $ \mathbb{R}i^! : \mathcal{ D}^{b}(X_{et},\Lambda) \rightarrow \mathcal{ D}^{b}(Z_{et},\Lambda)$ on the derived categories. Then Gabber purity implies the following quasi-isomorphism

\begin{center}
	$\mathbb{R}i^!C^{\bullet} \cong i^{\ast} C^{\bullet}(-c)[-2c]$
\end{center}
of chain complexes in $\mathcal{ D}^{b}(Z_{et},\Lambda)$ for any $C^{\bullet} \in \mathcal{ D}^{b}(X_{et},\Lambda)$.

\begin{lemma}\label{crux} Let $X/S$ be a Henselian regular local ring with $\sigma : k(x) \rightarrow X$, the closed point. Assume $ \sigma_{Z} : Z \hookrightarrow X$, $ \sigma_{Z'} : Z' \hookrightarrow X$ be regular closed subschemes (containing special fiber) such that $ Z \subset Z'$ and $c=codim(X,Z') =codim(X,Z) - 1$. Then the following morphism 
	\begin{center}
		$\sigma^{\ast}\sigma_{Z\ast}((\mathbb{R}\sigma_{Z}^{!})(\Lambda)) \rightarrow \sigma^{\ast}\sigma_{Z'\ast}((\mathbb{R}\sigma_{Z'}^{!})(\Lambda))$
	\end{center}
	is trivial in $\mathcal{ D}^{b}(k(x)_{et},\Lambda))$
\end{lemma}	

\begin{proof}
	We reduce the question to $Z'$ (which is  Henselian local because X is) and its codimension 1 closed subscheme $Z$. Denote $\mathbb{R}\sigma_{Z'}^{!}(\Lambda) $ by $\mathcal{ F}$ and consider the closed point  $\sigma: k(x) \xrightarrow{\sigma'} Z' \hookrightarrow X $. Then  purity for the closed immersion $\sigma_{Z/Z'} : Z \hookrightarrow Z'$ implies that $\mathbb{R}\sigma_{Z/Z'}^{!} \mathcal{ F} \cong \mathbb{R}\sigma_{Z}^{!}(\Lambda) \cong \mathcal{ F}(-1)[-2]$. Now by Lemma 6.6 of \cite{schmidt2019} and Remark \ref{zero}	$\sigma'^{\ast}\sigma_{Z/Z'\ast}((\mathbb{R}\sigma_{Z}^{!})(\Lambda)) \rightarrow \sigma'^{\ast}((\mathbb{R}\sigma_{Z'}^{!})(\Lambda))$ is trivial in $\mathcal{ D}^{b}(k(x)_{et},\Lambda))$ .
	
	We finish the proof by noting the isomorphisms $\sigma^{\ast}\sigma_{Z'\ast} \cong \sigma'^{\ast}$ and $ \sigma^{\ast}\sigma_{Z\ast} \cong \sigma'^{\ast}\sigma_{Z/Z'\ast}$.
\end{proof}

Since \'{e}tale cohomology is invariant for Henselian pairs, the previous lemma immediately yields the following corollary.

\begin{corollary}
	In the setting of Lemma \ref{crux}, the canonical morphism $\mathbb{R}\Gamma_Z(X_{et}, \Lambda) \rightarrow  \mathbb{R}\Gamma_{Z'}(X_{et}, \Lambda)$ is trivial.
\end{corollary}

Now we are in a position to prove the next theorem which will yield Bloch-Ogus theorem as its corollary. The key ingredients for the proof are Theorem \ref{gersternfornisnevich} (and  Proposition \ref{perf}) and Lemma \ref{crux}. We will merely sketch the proof as it follows the one given in \cite{schmidt2019}, once all the essential ingredients are in place.
\begin{theorem}\label{blochogusetale}
	Let $S$ be a $J$-$2$ Noetherian irreducible regular scheme of finite type.	Let $X/S$ be smooth, ${\rm dim}(X)=d$ and $C^\bullet$ an l.c.c.~complex in $\mathcal{D}_c^b(S_{\rm et},\Lambda)$.
	Then the Nisnevich Gersten complex $\mathcal{G}^\bullet(E(C^\bullet),n)$ is a flasque resolution of the Nisnevich sheafification $\mathbb{R}^n\varepsilon_*C^\bullet\vert_X$ of \'{e}tale cohomology with coefficients~$C^\bullet$.
	In particular, we get the exact sequence
	\begin{multline*}
		0 \rightarrow \mathbb{R}^n\varepsilon_\ast C^\bullet\vert_X \rightarrow \bigoplus_{z\in X^{(0)}}\mathfrak{j}_\ast {\rm H}^{n}(k(z),C^\bullet\vert_{k(z)})\rightarrow\dots\\
		\dots \rightarrow \bigoplus_{z\in X^{(d)}}\mathfrak{j}_\ast {\rm H}^{n-d}(k(z),C^\bullet\vert_{k(z)}(-d)) \rightarrow 0.
	\end{multline*}
\end{theorem}
\begin{proof}
	First we need to say $E(K^{\bullet})$ is $\A^1$-local to be able to use Proposition \ref{perf}. This follows from Lemma 6.3 of \cite{schmidt2019}. Next we extend  Lemma \ref{crux} to any complex $C^{\bullet}\in \sD_c^b(S_{et},\Lambda)$, that is, the morphism	$\sigma^{\ast}\sigma_{Z\ast}((\mathbb{R}\sigma_{Z}^{!})(C^\bullet|_{X})) \rightarrow \sigma^{\ast}\sigma_{Z'\ast}((\mathbb{R}\sigma_{Z'}^{!})(C^\bullet|_{X}))$ is trivial. Hence by Proposition \ref{perf}, $\mathcal{G}^\bullet(E(C^\bullet),n)$ is a flasque resolution of  $\mathbb{R}^n\varepsilon_*C^\bullet\vert_X$. This proves the first part of the theorem.
	
	Then one proves $ \mathfrak{j}^\ast E(C^\bullet)_{Z/X}^{n+s} \cong {\rm H}^{n-s}(k(z),C^\bullet\vert_{k(z)}(-s)) $ (See {\cite[Proof of Theorem 6.8]{schmidt2019}} for details). As $\mathcal{G}^s(E(C^\bullet),n) = \bigoplus_{z\in X^{(s)}}\mathfrak{j}_\ast\mathfrak{j}^\ast E(C^\bullet)_{Z/X}^{n+s}$, this concludes the proof.

\end{proof}

\noindent Theorem \ref{blochogusetale} immediately yields  Theorem \ref{cor: Bloch Ogus for etale cohomology of Nisnevich local schemes} after taking the Nisnevich stalks of the spectrum.

\begin{remark}
	In fact, Theorem \ref{cor: Bloch Ogus for etale cohomology of Nisnevich local schemes} holds for any $\A^1$-invariant cohomology theory that satisfies purity and admits a reasonable notion of Chern classes. The details and precise formulation of this observation will be developed in a future work.
\end{remark}

\vspace{1.5cm}
\bibliographystyle{abbrv}
\bibliography{biblio}

\vspace{2cm}

\begin{center}
Neeraj Deshmukh, IISER Pune, Dr. Homi Bhabha Road, Pashan, Pune : 411008, INDIA\\
email: neeraj.deshmukh@students.iiserpune.ac.in\\
\vspace{0.3cm}
Girish Kulkarni, Fachgruppe Mathematik/Informatik, Bergische Universität Wuppertal, Gaußstraße 20, 42119 Wuppertal, GERMANY\\
email: kulkarni@uni-wuppertal.de\\
\vspace{0.3cm}
Suraj Yadav, IISER Pune, Dr. Homi Bhabha Road, Pashan, Pune : 411008, INDIA\\
email: surajprakash.yadav@students.iiserpune.ac.in\\
\end{center}
\end{document}